\documentclass[11pt]{amsart}

\usepackage[utf8x]{inputenc} 
\usepackage[a4paper]{geometry} 
\usepackage{enumerate}  
\usepackage{amsmath} 
\usepackage{amssymb}
\usepackage{color}
\usepackage{multirow}
\usepackage{hyperref}  
\usepackage{verbatim}
                 \hypersetup{ pdfborder={0 0 0}, 
                              colorlinks=true, 
                              linktoc=page,
                              pdfauthor={Michael Hoff and Giovanni Staglian{\`o}}, 
                              pdftitle={New examples of rational Gushel-Mukai fourfolds}} 
\definecolor{Gray}{gray}{0.9}                            
\usepackage[arrow,curve,matrix,tips,frame]{xy}

\theoremstyle{plain} 
\newtheorem{proposition}{Proposition}[section] 
\newtheorem{theorem}[proposition]{Theorem} 
\newtheorem{lemma}[proposition]{Lemma}

\theoremstyle{definition}

\newtheorem{example}[proposition]{Example} 
\theoremstyle{remark} 
\newtheorem{remark}[proposition]{Remark}

\newcommand{\PP}{{\mathbb{P}}}          
\newcommand{\GG}{{\mathbb{G}}}

\numberwithin{equation}{section}

\title[New examples of rational Gushel-Mukai fourfolds]{{New examples of rational Gushel-Mukai fourfolds}}

\author[M. Hoff]{Michael Hoff} 
\address{Universit\"at des Saarlandes, Campus E2 4, D-66123 Saarbr\"ucken, Germany}
\email{\href{mailto:hahn@math.uni-sb.de}{hahn@math.uni-sb.de}} 

\author[G. Staglian\`o]{Giovanni Staglian\`o} 
\address{Dipartimento di Matematica e Informatica, Universit\`{a} degli Studi di Catania, Viale A. Doria 5, 95125 Catania, Italy}
\email{\href{mailto:giovannistagliano@gmail.com}{giovannistagliano@gmail.com}} 

\date{\today} 

\begin{document}

\begin{abstract} 
We construct new examples of rational Gushel-Mukai fourfolds,
giving  more evidence for the analog of the Kuznetsov Conjecture for cubic fourfolds:
a Gushel--Mukai fourfold is rational if and only if it admits an \emph{associated K3 surface}.
\end{abstract}

\maketitle

\section*{Introduction} 
A Gushel-Mukai fourfold is a smooth prime Fano fourfold $X\subset \PP^8$ of degree $10$ and index $2$ (see \cite{mukai-biregularclassification}).
These fourfolds are parametrized by a coarse moduli space $\mathcal M_4^{GM}$ of dimension $24$ (see \cite[Theorem 5.15]{DK3}), and
the general fourfold $[X]\in \mathcal M_4^{GM}$ is a smooth quadratic section of a smooth hyperplane section of the Grassmannian $\mathbb{G}(1,4)\subset\PP^9$ of lines in $\PP^4$. 

In \cite{DIM} (see also \cite{DK1,DK2,DK3}), following Hassett's analysis of cubic fourfolds (see \cite{Hassett,Has00}), the authors 
studied Gushel-Mukai fourfolds via Hodge theory and the period map.
In particular, they showed that 
inside $\mathcal M_4^{GM}$ there is 
 a countable union $\bigcup_d \mathcal{GM}_d 
 $
 of (not necessarily irreducible) hypersurfaces
parametrizing \emph{Hodge-special} Gushel-Mukai fourfolds, that is, 
fourfolds that contain a surface whose cohomology class does not come from the Grassmannian $\mathbb{G}(1,4)$.
The index $d$ is called the \emph{discriminant} of the fourfold and it runs over all positive integers 
congruent to $0,2$, or $4$ modulo $8$ (see \cite{DIM}).
However, as far as the authors know, explicit geometric descriptions of Hodge-special Gushel-Mukai fourfolds in $\mathcal{GM}_d$ are unknown for $d>12$. 
In Theorem~\ref{mainthm1}, we shall provide such a description when $d=20$.

As in the case of cubic fourfolds, all Gushel-Mukai fourfolds are unirational. Some rational examples are classical and easy to construct, but no examples have yet been proved to be irrational.
Furthermore, there are values of the discriminant $d$ such that a fourfold in $\mathcal{GM}_d$ admits an \emph{associated K3 surface} of degree $d$. For instance, this occurs 
for $d=10$ and $d=20$.
The hypersurface $\mathcal{GM}_{10}$ has two irreducible components,  
and the general fourfold in each of these two components is rational (see \cite[Propositions~7.4 and 7.7]{DIM} and Examples~\ref{exa1} and \ref{exa2}).
Some of these fourfolds were already studied by Roth in \cite{Roth1949}.
As far as the authors know, there were no other known examples of rational Gushel-Mukai fourfolds. In Theorem~\ref{RatGM},
we shall provide new examples of rational Gushel-Mukai fourfolds that belong to $\mathcal{GM}_{20}$. 

A classical and still open question in algebraic geometry 
is the rationality of smooth cubic hypersurfaces in $\PP^5$ (cubic fourfolds for short). 
An important conjecture, known as \emph{Kuznetsov's Conjecture} (see \cite{AT, kuz4fold, kuz2, Levico}) asserts that a cubic fourfold is rational if and only if it admits an associated K3 surface in the sense of Hassett/Kuznetsov.
This condition can be expressed by saying that the rational cubic fourfolds are parametrized by a countable union $\bigcup_d \mathcal C_d$ of irreducible  hypersurfaces inside the $20$-dimensional coarse moduli space of cubic fourfolds, where $d$ runs over the so-called \emph{admissible values} (the first ones are $d=14,26,38,42,62$).
The rationality of cubic fourfolds in $\mathcal C_{14}$ was proved by Fano in \cite{Fano} (see also \cite{BRS}), while rationality in the case of $\mathcal C_{26}$ and $\mathcal C_{38}$ was proved in \cite{RS1}. Very recently, in \cite{RS3}, rationality was also proved in the case of $\mathcal C_{42}$.
The proof of this last result shows a close relationship between cubic fourfolds in $\mathcal C_{42}$ and the Gushel-Mukai fourfolds in $\mathcal{GM}_{20}$ constructed in this paper.
This beautiful geometry was discovered with the help of \emph{Macaulay2} \cite{macaulay2}.

\subsection*{Acknowledgements}
This work started during the \emph{Macaulay2} workshop 
held at the Saarland University in September 2019.
We would like to thank the organizers for bringing us together as well as Hoang Le Truong. 
We also wish to thank Francesco Russo and Olivier Debarre for very relevant suggestions.

\section{Generality on Gushel-Mukai fourfolds}\label{generalitiesGM}
In this section, we recall some general facts about Gushel-Mukai fourfolds 
which were proved in \cite{DIM} (see also \cite{DK1,DK2,DK3}).

A Gushel-Mukai fourfold $X\subset \PP^8$, GM fourfold for short,
is a degree-$10$ Fano fourfold with $\mathrm{Pic}(X) = \mathbb{Z}[ \mathcal{O}_X(1)]$ 
and $K_X\in|\mathcal{O}_X(-2)|$.
Equivalently,
$X$ is  a quadratic section of a {$5$}-dimensional linear section 
of the cone in $\PP^{10}$ over the Grassmannian $\GG(1,4)\subset \PP^9$ of lines in $\PP^4$.
There are two types of GM fourfolds:
\begin{itemize}
 \item quadratic sections of hyperplane
sections of $\mathbb G(1,4)\subset\PP^9$ (\emph{Mukai or ordinary fourfolds, \cite{mukai-biregularclassification}});
\item double covers of $\mathbb G(1,4)\cap\PP^7$ branched along its intersection with a quadric 
(\emph{Gushel fourfolds}, \cite{Gu}). 
\end{itemize}
 There exists a $24$-dimensional coarse moduli space $\mathcal M_4^{GM}$ of GM fourfolds, where the locus of Gushel fourfolds is of codimension $2$. Moreover, we have a \emph{period map} $\mathfrak{p}:\mathcal M_4^{GM}\to\mathcal{D}$ to a $20$-dimensional quasi-projective variety $\mathcal D$,
 which is dominant with irreducible $4$-dimensional fibers (see \cite[Corollary 6.3]{DK3}). 

For a very general GM fourfold $[X]\in \mathcal M_4^{GM}$, the natural inclusion 
\begin{equation}\label{naturalInclusion}
A(\GG(1,4)) := H^4(\mathbb G(1,4),\mathbb Z)\cap H^{2,2}(\mathbb G(1,4))\subseteq A(X) := H^4(X,\mathbb Z)\cap H^{2,2}(X)
\end{equation}
of middle Hodge groups is an equality.
A GM fourfold $X$ is said to be \emph{Hodge-special} if the inclusion \eqref{naturalInclusion} is strict. This means that the fourfold $X$ contains a surface whose cohomology class ``does not come'' from the Grassmannian $\GG(1,4)$.
Hodge-special GM fourfolds are parametrized by a countable union of hypersurfaces 
$\bigcup_d \mathcal{GM}_d\subset \mathcal M_4^{GM}$, labelled 
by the positive integers $d\equiv 0,2$, or $4$ (mod $8$) (see \cite[Lemma~6.1]{DIM}). The image 
$\mathcal{D}_d=\mathfrak{p}(\mathcal{GM}_d)$ is a hypersurface in $\mathcal{D}$,
which is irreducible if $d\equiv 0$ (mod $4$), and has two irreducible components $\mathcal D_d'$ and $\mathcal D_d''$
if  $d\equiv 2$ (mod $8$) (see \cite[Corollary~6.3]{DIM}). The same holds true for $\mathcal{GM}_d$.

In some cases, the value of $d$ can be explicitly computed from the geometry of Hodge-special GM fourfolds (see \cite[Section~7]{DIM}). Indeed, let $X\subset \PP^8$ be an ordinary  GM fourfold containing a smooth surface $S$ such that $[S]\in A(X)\setminus A(\mathbb G(1,4))$.
We may write $[S]=a\sigma_{3,1}+b\sigma_{2,2}$ in terms of Schubert cycles in $\mathbb G(1,4)$ for some integers $a$ and $b$. 
We then have $[X]\in \mathcal{GM}_d$, where $d$ is the absolute value of the determinant (or {\it discriminant}) of the intersection matrix 
in the basis $(\sigma_{1,1|X}, \sigma_{2|X}-\sigma_{1,1|X}, [S])$. 
That is 
\begin{equation}\label{discriminant} 
d=\left|\det  \begin{pmatrix}
     2&0&b\\
     0&2&a-b\\
     b&a-b&(S)_X^2\end{pmatrix}\right| = \left|
     4 (S)_X^2-2(b^2+(a-b)^2)\right|, 
\end{equation}
where 
\begin{equation}\label{doublepoints}
(S)_X^2=3a+4b+2K_S\cdot \sigma_{1|S}+2K_S^2-12\chi(\mathcal O_S).
\end{equation}

For some values of $d$, the non-special cohomology of the GM fourfold $[X]\in \mathcal{GM}_d$
looks like the primitive cohomology of a K3 surface. In this case, as in the case of cubic fourfolds,
one says that $X$ has an associated K3 surface. The first values of $d$ that satisfy 
the condition for the existence of an associated K3 surface are: $2$, $4$, $10$, $20$, $26$, $34$.
We refer to \cite[Section~6.2]{DIM} for precise definitions and results.

In Examples~\ref{exa1} and \ref{exa2} below, we recall the known examples of rational GM fourfolds,
which all have discriminant $10$.
In Section~\ref{second}, 
we shall construct rational GM fourfolds of discriminant $20$.
\begin{example}\label{exa1}
  A \emph{$\tau$-quadric} surface in $\mathbb{G}(1,4)$ is a linear section of 
  $\mathbb{G}(1,3)\subset\mathbb{G}(1,4)$; its class is $\sigma_{1}^2\cdot \sigma_{1,1} = \sigma_{3,1}+ \sigma_{2,2}$.
 In \cite[Proposition~7.4]{DIM}, it was proved that the closure $D_{10}'\subset \mathcal M_4^{GM}$ of the family of fourfolds containing a $\tau$-quadric surface is the irreducible hypersurface $\mathfrak{p}^{-1}(\mathcal{D}_{10}')$, and that the general member of $D_{10}'$ is rational.
 Furthermore, they are all rational by \cite{KontsevichTschinkelInventiones} or \cite[Theorem 4.15]{DK1}.

 In \cite[Theorem~5.3]{RS3}, a different description of $D_{10}'$ and another proof of the rationality of its general member were given.

The rationality for a general fourfold $ [X] \in D_{10}' $ 
also follows from the fact that
a $ \tau $-quadric surface $ S $, inside the unique del Pezzo fivefold $ Y
\subset \PP^8 $ 
containing $ X $, admits \emph{a congruence of $ 1 $-secant lines}, that is, through
the general point of $ Y $, there passes just one line contained in $ Y $ which intersects $ S $.
 \end{example}

\begin{example} \label{exa2}
A quintic del Pezzo surface is a two-dimensional linear section of $\mathbb{G}(1,4)$;
its class is $\sigma_{1}^4 = 3\sigma_{3,1} + 2 \sigma_{2,2}$.
 In \cite[Proposition~7.7]{DIM}, it was proved that the 
 closure $D_{10}''\subset \mathcal M_4^{GM}$ of the family 
  of fourfolds containing a quintic del Pezzo surface is the irreducible hypersurface
  $\mathfrak{p}^{-1}(\mathcal{D}_{10}'')$. 

The proof of the rationality of a general fourfold 
 $[X]\in D_{10}''$ is very classical. 
 Indeed in \cite{Roth1949}, Roth remarked that the projection from the linear span of a quintic del Pezzo surface  contained in $ X $ induces a dominant map
 \[ \pi : X\dashrightarrow \PP^2 \]
 whose generic fibre is a quintic del Pezzo surface.
 By a result of Enriques (see \cite{Enr,EinSh}),
 a quintic del Pezzo surface defined over an infinite field $ K $ is $ K $-rational. Thus, 
 the fibration $ \pi $ admits a rational section and $ X $ is rational.
\end{example}

\section{A Hodge-special family of Gushel-Mukai fourfolds}\label{second}
Let $S\subset\PP^8$ be 
the image of $\PP^2$ via the linear system of quartic curves through three simple points and one double point in general position.
Then $S$ is a smooth surface of degree $9$ and sectional genus $2$ cut out in $\PP^8$ by $19$ quadrics.

\begin{lemma}\label{surfInG14} Let $S\subset\PP^8$ be a rational surface of degree $9$ and sectional genus $2$ as above.
Then $S$ can be embedded in a smooth del Pezzo fivefold $Y=\GG(1,4)\cap \PP^8$ 
such that 
in the Chow ring of $\GG(1,4)$, we have 
 \begin{equation}\label{classAB}
 [S]=6\,\sigma_{3,1} + 3\, \sigma_{2,2} .
 \end{equation}
 Moreover, there exists an irreducible component of the Hilbert scheme parameterizing such surfaces in $Y$ 
 which is generically smooth of dimension $25$.
\end{lemma}
\begin{proof}
 Using \emph{Macaulay2} \cite{macaulay2} (see Section~\ref{computations}), we constructed a specific example of a surface $S\subset \PP^8$ as above  
which is embedded in a del Pezzo fivefold $Y\subset\PP^8$ and satisfies \eqref{classAB}.
Moreover we verified in our example that 
$h^0(N_{S,Y})=25$ and $h^1(N_{S,Y})=0$. Thus, $[S]$ is a smooth point in the corresponding Hilbert scheme 
$\mathrm{Hilb}_Y^{\chi(\mathcal O_S(t))}$
of subschemes of $Y$,
and the unique irreducible component of $\mathrm{Hilb}_Y^{\chi(\mathcal O_S(t))}$ containing $[S]$ has dimension~$25$.
\end{proof}

\begin{remark}\label{rem0}
After our construction, in a preliminary version of this paper, of an explicit example of a surface 
as in Lemma~\ref{surfInG14}, \cite[Section~4]{RS3} provided an explicit geometric description of an irreducible $25$-dimensional family 
of these surfaces inside a del Pezzo fivefold, confirming the claim of Lemma~\ref{surfInG14}.
\end{remark}

\begin{theorem}\label{mainthm1}
Inside $\mathcal M_4^{GM}$, the closure  $D_{20}$
of the family of GM fourfolds containing a surface $S\subset\PP^8$ as in Lemma~\ref{surfInG14} is the irreducible hypersurface ${\mathfrak p}^{-1}(\mathcal{D}_{20})$.
\end{theorem}
\begin{proof}
Let $Y=\GG(1,4)\cap\PP^8$ be a fixed smooth del Pezzo fivefold and 
let $\mathcal S$ be the $25$-dimensional irreducible family of rational surfaces 
$S\subset Y$ of degree $9$ and sectional genus $2$ described in Lemma~\ref{surfInG14}.
Let $\mathcal{GM}_{Y} = \mathbb{P}(H^0(\mathcal O_{Y}(2)))$ denote the family of GM fourfolds contained in $Y$, that is, the family of quadratic sections 
of $Y$. The dimension of $\mathcal{GM}_{Y}$ is $ h^0(\mathcal O_{\PP^8}(2)) - h^0(\mathcal I_{Y,\PP^8}(2)) - 1 = 39$.

Consider the incidence correspondence
$$I=\overline{\{([S], [X])\;:\; S\subset X \subset Y \}}\subset \mathcal S\times\mathcal{GM}_{Y},$$ and let
$$\xymatrix{
 &I\ar[dl]_{p_1}\ar[dr]^{p_2}&&\\
\mathcal S&&\mathcal{GM}_{Y}&}
$$
be the two natural projections.
Then $p_1$ is a surjective morphism and, for $[S]\in \mathcal S$ general, 
the fibre $p_1^{-1}([S]) \simeq \PP(H^0(\mathcal I_{S,Y}(2)))$ is irreducible 
of dimension $h^0(\mathcal I_{S,\PP^8}(2)) - h^0(\mathcal I_{Y,\PP^8}(2)) - 1 = 13$.
It follows that $I$ has a unique irreducible component $I^0$ that dominates $\mathcal S$ and that component has dimension $25 + 13 = 38$.

Using \emph{Macaulay2} (see \cite{M2files}), we verified in a specific example of a GM fourfold $X$ 
containing a surface $[S]\in \mathcal S$ 
that $H^0(N_{S,X}) = 0$. By semicontinuity, we deduce that $p_2$ is a generically finite morphism onto its image and that $p_2(I^0)$ has dimension $38$. It is therefore a hypersurface in $\mathcal{GM}_Y$. Since all smooth hyperplane sections of the Grassmannian $\mathbb{G}(1,4)\subset\PP^9$  are projectively equivalent, $\mathcal{GM}_Y$ dominates $\mathcal M_4^{GM}$ and the fourfolds $X$ that we have constructed form an irreducible hypersurface in $\mathcal M_4^{GM}$.

Finally, by applying \eqref{discriminant} and \eqref{doublepoints}, we get 
that a general such $[X]$ 
lies in $\mathfrak{p}^{-1}(\mathcal{D}_{20})$. This proves the theorem.
\end{proof}

\begin{theorem}\label{RatGM}
Every GM fourfold belonging to the family $D_{20}$ described in Theorem~\ref{mainthm1} is rational.
\end{theorem}
\begin{proof}
Let $Y\subset\PP^8$ be a del Pezzo fivefold and let $S\subset Y$ 
 be a general rational surface of degree $9$ and sectional genus $2$ 
belonging to the $25$-dimensional family described in Lemma~\ref{surfInG14} and Remark~\ref{rem0}.

 The restriction to  $Y$ of 
 the linear system of cubic hypersurfaces
with double points along $S$ gives a dominant rational map 
\begin{equation*} 
 \psi:Y\dashrightarrow \PP^4
\end{equation*}
whose general fibre is an irreducible conic curve which intersects $S$ at three points.
Thus $S$ admits inside $Y$ a \emph{congruence of $3$-secant conic curves}. 
This implies that the restriction of $\psi$ to a general GM fourfold $X$ 
containing $S$ and contained in $Y$ 
is a birational map to $\PP^4$.

The existence of the congruence of $3$-secant conics
can be also verified as follows. The linear system of quadrics through $S$ 
induces a birational map 
\[
\phi:Y\dashrightarrow Z\subset \PP^{13}
\]
 onto a fivefold $Z$ of degree $33$ and cut out by $21$ quadrics.
 Let $p\in Y$ be a general point. Then one sees that through $\phi(p)$ there pass 
 $7$ lines contained in $Z$. Of these, $6$ are the images of the lines passing through 
 $p$ and which intersect $S$, while the remaining line come from a single $3$-secant conic 
 to $S$ passing through $p$.

The claim about the rationality of \emph{every} $[X]\in D_{20}$ follows 
from the rationality of a general $[X]\in D_{20}$ and from
the main result in \cite{KontsevichTschinkelInventiones} or from \cite[Theorem 4.15]{DK1}.
\end{proof}

\begin{remark}
 The inverse map of the birational map $\psi:X\dashrightarrow \PP^4$ 
 described in the proof of Theorem~\ref{RatGM} is defined 
 by the linear system of hypersurfaces of degree $9$ 
 having double points along 
 an 
 internal projection to $\PP^4$ 
 of a smooth surface $T\subset \PP^5$ of degree $11$ and sectional genus $6$
 cut out by $9$ cubics. This surface $T$ is  
 an internal triple projection of 
 a smooth minimal K3 surface of degree $20$ and  genus $11$ in $\PP^{11}$.
 
 Actually, this was the starting point for this work.
 In fact, from the results of \cite{RS3}, 
 we suspected that a triple internal projection of a 
 minimal K3 surface of degree $20$ and genus $11$ 
 could be related to a GM fourfold of discriminant~$20$.
\end{remark}

\section{Explicit computations}\label{computations}
In the proof of Lemma~\ref{surfInG14}, we claimed that 
there exists an example of a rational surface $S\subset\PP^8$ of degree $9$ and sectional genus $2$  
which is also embedded in $\GG(1,4)$ and satisfies $[S]=6\,\sigma_{3,1} + 3\, \sigma_{2,2}$. In an ancillary file (see \cite{M2files}), we provide the explicit homogeneous 
ideal of such a surface which contains the ideal generated by the Pl\"{u}cker relations of $\GG(1,4)$. The class $[S]$ in terms 
of the Schubert cycles $\sigma_{3,1}$ and $\sigma_{2,2}$ can be easily calculated using, for instance, the \emph{Macaulay2} package \emph{SpecialFanoFourfolds}.

In the following, we explain the main steps of the procedure we followed to construct the surface in $\GG(1,4)$.
We start by taking a general nodal hyperplane section 
of a smooth Fano threefold of degree $22$ and sectional genus $12$ in $\mathbb{P}^{13}$ (see \cite{Mukai1983,schreyer_2001}).
The projection of this surface from its node yields 
a smooth K3 surface $T\subset\PP^{11}$ of degree $20$ and sectional genus $11$ which contains a conic 
(see \cite{Kapustka_2018}).
Then we take a general triple projection of $T$ in $\PP^5$, which 
  is a smooth surface of degree $11$ and sectional genus $8$ (this follows  from 
 \cite[Proposition 4.1]{Voisin} and \cite[Theorem 10]{FS} in the case when the K3 surface $T$ is general).
 Let $T'\subset \PP^4$ be a general internal projection of this surface in $\PP^5$. Then 
 $T'$ is a singular surface of degree $10$ and sectional genus $8$, cut out by 
 $13$ quintics. The linear system of hypersurfaces of degree $9$ having double points along $T'$ 
 gives a birational map 
 $
 \eta: \PP^4\dashrightarrow X\subset\PP^8
 $
onto a GM fourfold $X$, whose 
 inverse map is defined by the restriction to $X$ of the linear system of cubic hypersurfaces having double points 
 along a smooth surface $S\subset X$ of degree $9$ and sectional genus $2$.
Finally, to determine explicitly the surface $S$, one can exploit the fact that 
 the general quintic hypersurface corresponds via $\eta$ to the general quadric hypersurface (inside $X$) containing $S$.
 Indeed,
behind the scenes, we have an occurrence of a flop, similar to the 
\emph{Trisecant Flop} considered in \cite{RS3}. In particular, 
we have a commutative diagram 
 $$\xymatrix{
 &M&&\\
\PP^4 \ar@{-->}[ur]^{m_1} \ar@{-->}[rr]^{\eta}&& X \ar@{-->}[ul]_{m_2}&}
$$
 where 
 $m_1$ and $m_2$ are the birational maps defined, respectively, 
 by the linear system of quintics through $T'$ and 
 by the linear system of quadrics through $S$. Moreover, $M$ is a fourfold of degree $33$ in $\PP^{12}$ cut out by $21$ quadrics. 
 
% \begin{remark}
% The Hilbert scheme parameterizing surfaces $T'\subset\PP^4$ as above is generically smooth of dimension $47= \dim\mathrm{PGL}(5,\mathbb C) + 23$. So a naive parameter count would suggest that we can get the general GM fourfold in the hypersurface $D_{20}$ if we start  the above procedure by taking a general K3 surface of genus $11$, rather of  one containing a conic. However, we are not able to construct the general K3 surface of genus $11$.
% \end{remark}
 
%\bibliographystyle{amsalpha}
%\bibliography{bibliography}

\providecommand{\bysame}{\leavevmode\hbox to3em{\hrulefill}\thinspace}
\providecommand{\MR}{\relax\ifhmode\unskip\space\fi MR }
\providecommand{\MRhref}[2]{%
  \href{http://www.ams.org/mathscinet-getitem?mr=#1}{#2}
}
\providecommand{\href}[2]{#2}

\end{document}